\documentclass[11pt]{article}
\usepackage{a4,amssymb,amsthm,amsmath} 
\nonstopmode

\newcommand{\R}{{\mathbb R}} 

\newcommand{\abs}[1]{{\left| {#1} \right|}} \newcommand{\p}[1]{{\left(
      {#1} \right)}} 
\newcommand{\norm}[1]{\left \Vert {#1} \right \Vert}
\newcommand{\ip}[2]{{\left< {#1},{#2} \right>}}
\newcommand{\Oh}[1]{{O \p{#1}}}

\renewcommand{\Re}{\operatorname{Re}}
\newcommand{\re}{\operatorname{Re}}
\begin{document}
\def\cprime{$'$}

\theoremstyle{plain}
\newtheorem{thm}{Theorem}
\newtheorem{cor}{Corollary}
\newtheorem{lem}{Lemma}
\newtheorem{prop}{Proposition}

\theoremstyle{definition}
\newtheorem{example}{Example}
\newtheorem{rem}{Remark} 
\newtheorem{conj}{Conjecture}
\renewcommand{\theconj}{\hskip -4pt}
\newtheorem{prob}{Problem}
\newtheorem{vanlem}{Vanishing Lemma} 
\newtheorem{pecher}{Pechersky's rearrangement theorem}
\renewcommand{\thevanlem}{\hskip -4pt}
\renewcommand{\thepecher}{\hskip -4pt}

\author{Johan Andersson\thanks{Department of Mathematics, Stockholm University, SE-106 91 Stockholm SWEDEN, Email: {\texttt {johana@math.su.se}}.} \\ \\ {\em Dedicated to the memory of K. Ramachandra}}

\date{}

\title{On a problem of Ramachandra and approximation of functions by Dirichlet polynomials with bounded coefficients}

\maketitle
  
\begin{abstract}
 We prove effective results on when a function can be approximated by a  Dirichlet polynomial with bounded coefficients. Assuming that $\Phi(n)$ is an increasing function we prove that the set of polynomials
\begin{gather} \notag
  \left \{ \sum_{n=2}^N a_n n^{it-1}: N \geq 2, |a_n| \leq \Phi(n) \right \}, \\ \intertext{is dense in  $L^2(0,H)$ if and only if}
  \sum_{n=2}^\infty \frac{\log \Phi(n)} {n \log^2 n}  = \infty. \tag{$*$}
\end{gather}
We also prove variants of this result for generalized Dirichlet polynomials. The main tools are theorems of Paley and Wiener related to quasianalyticity and the Pechersky rearrangement theorem. 
 We use this result to give precise conditions on when a conjecture of Ramachandra is true and when it is false. We prove that whenever $\Phi(n)$ is a positive increasing function then
\begin{gather*} 
  \lim_{N \to \infty} \min_{|a_n| \leq \Phi(n)} \int_0^H \abs{1+\sum_{n=2}^N a_n n^{it-1}}^2 dt =0, 
\end{gather*}
if and only if the sum $(*)$ is divergent.
This has applications on lower bounds for moments of the Riemann zeta-functions in short intervals close to $\Re(s)=1$, and to questions of Universality for zeta-functions on and close to their abscissa of convergence.   
\end{abstract}

 \tableofcontents 

\section{A problem of Ramachandra}
\subsection{Ramachandra's original problem}
Ramachandra stated the following conjecture (\cite{Ramachandra2}, for related problems see
 \cite{BalRam})
\begin{conj} 
 Does there for each  $\delta>0$  exist a $H>0$ such that 
$$
\int_0^H \left|1+ \sum_{n=2}^N a_n n^{it} \right|^2 dt > \delta, 
$$
for all integers $N \geq 2$ and complex numbers $a_n$?
\end{conj}
If true, Ramachandra \cite{Ramachandra2} proved that it would have important applications on the Riemann zeta-function.

However, we proved  that this conjecture was  false in \cite{Andersson}. Our  first counterexample used the Szasz-M\"untz
theorem and gives no estimates on the growth of $a_n$. In private communication Ramachandra asked whether Conjecture 1  is still false if the $a_n$ can be assumed to be polynomially increasing. In our second counterexample we proved that we  can assume $$\abs{ a_n} \leq n^{c-1},$$ for any $c>0$. Our proof used  universality properties of the Riemann zeta function.

\subsection{Sketch of proof}
The key point is that the Riemann zeta-function can be estimated by a Dirichlet polynomial
\begin{gather} \notag \left|\zeta(\sigma+it)-\sum_{n=1}^N n^{-\sigma-it} \right| < \epsilon/3, \qquad N \asymp t, \, \, 1/2<\sigma<1.  \\
\intertext{Then we use universality (on the compact subset $K=[\sigma,\sigma+iH]$) to find a $T$ such that} 
\label{AB} \left| \zeta(\sigma+iT+it)-\epsilon/3 \right|<\epsilon/3,  \\ \intertext{for $0 \leq t \leq H$. It follows from  the triangle inequality that}
\notag \left| \sum_{n=1}^N a_n n^{it} \right|<\epsilon,  \qquad 0<t<H,
\end{gather}
for $a_n=n^{iT-\sigma}$, from which a negative answer to Ramachandra's problem follows.

This proof was given in \cite{Andersson} and also appears in Steuding \cite[Section 10.4]{Steuding} as an example of applications of Universality. We remark that the proof can be somewhat simplified if we may allow the function that we approximate to be $0$ and not $\epsilon/3$ in  \eqref{AB}. The reason why we use $\epsilon/3$ is because the version of Voronin universality theorem we used was a version proved by Bagchi \cite{Bagchi}, which required that the function we approximate is nonvanishing. In a recent paper \cite{Andersson2} we proved however, that in general we may allow the function to have zeroes in the interval.

\subsection{A refined Ramachandra problem}
Ramachandra's motivation probably lies somewhere in the various inequalities he and Balasubramanian  actually did prove. In particular in the context  of Weak Titchmarsh-series (see e.g. Ramachandra \cite{Ramachandra}) they proved that the conjecture is true if $\abs{a_n} \ll (\log n)^C$. Therefore, a natural question is.  For which increasing functions $\Phi(n)$ and $H>0$ is the following true:
\begin{gather*}\lim_{N \to \infty} \min_{|a_n| \leq \Phi(n)} \int_0^H \left|1+\sum_{n=2}^N a_n n^{it-1} \right|^2 dt > 0? \tag{*} \end{gather*}
As we have indicated
\begin{enumerate}
 \item (*) is true for $\Phi(n) \ll (\log n)^C$, and $C>0$.
 \item (*) is false for $\Phi(n) \gg  n^\delta$, and $\delta>0$.
\end{enumerate}
What about the intermediate cases? Our main aim is to solve this problem completely in terms of growth of the function $\Phi(n)$.

\subsection{A  solution to the refined Ramachandra problem}
 The following result  was first presented in a weaker form (for some $H>0$ instead of for all $H>0$) at the Zeta-Function-Days in Seoul, September 2009 and in its final form at the Tata institute in Mumbai one month later gives a final answer to the refined version of Ramachandra's problem:
\begin{thm} Suppose  $\Phi(n)$ is an increasing positive function and $H>0$. Then 
\begin{gather*}\lim_{N \to \infty} \min_{|a_n| \leq \Phi(n)} \int_0^H \left|1+\sum_{n=2}^N a_n n^{it-1} \right|^2 dt > 0,
  \\ \intertext{if and only if} 
  \sum_{n=2}^\infty  \frac {\log{\Phi(n)}}{n \log^2 n}<\infty.  
\end{gather*}
\end{thm}
Like Ramachandra's original conjecture, this result will have applications on the Riemann zeta-function, although somewhat weaker ones. For example in a forthcoming paper  \cite{Andersson3} we prove that
\begin{gather} \label{rrr}
 \inf_T \int_T^{T+\delta} |\zeta(1+it)| dt=\frac{\pi^2 e^{-\gamma}} {24} \delta^{2}+ \Oh{\delta^4}, \qquad (\delta>0),
\end{gather}
We remark that this gives a positive lower bound that is independent of $T$. This implies that the Riemann zeta-function is not universal on the line $\Re(s)=1$, since if the Riemann zeta-function was universal on $\Re(s)=1$ it should be possibly to  approximate an arbitrarily small function by the Riemann zeta-function on that line. It is clear that Theorem 1 will also give us a lower bound independent of $T$ in \eqref{rrr}, by choosing $\Phi(n)=1$ and approximating the Riemann zeta-function by a truncated Dirichlet polynomial. While the constant will not be explicitly given in $\delta$, and thus not give as sharp result as \eqref{rrr},   Theorem 1 do allow us to prove  corresponding lower bounds, and non universality on curves $\gamma(t)=\omega(t)+it$ whenever $\omega(t)$ tends to $1$ sufficiently fast when $t \to \infty$.

We will also generalize Theorem 1  to more general Dirichlet polynomials. For example the sum over integers $n$  can be replaced with sum over primes $p$ or over shifted integers $n+\alpha$. Also we may allow coefficients that can be quite general. Simple interesting cases includes  divisor functions and Fourier coefficients of Maass wave forms. We will investigate this more carefully later in the paper. For now we just remark that a version of Theorem 1 for the shifted integers $n+\alpha$   implies the following result:
\begin{gather} \label{aj}
  \int_{T}^{T+\delta} \abs{\zeta\p{1+it,\alpha}}dt \geq C_\delta>0,
\end{gather}
from which the fact that the Hurwitz zeta-function is not universal on the line $\Re(s)=1$ follows. The method we use in \cite{Andersson3} for the Riemann zeta-function requires some multiplicative property, such as the function has an Euler product. The method from this paper does not require any such result however, although the proof in this paper will be somewhat more indirect and  not give explicit estimates of $C_\delta$ such as  Eq. \eqref{rrr}. 
For an effective version of the method used in this paper  and some explicit estimates of $C_\delta$ in \eqref{aj},  see  our forthcoming paper \cite{Andersson15}.

\section{Lower bounds for Dirichlet polynomials}

\subsection{A lower bound}
We will first prove a result that implies the lower bound in Theorem 1. We choose to state the theorem for general Dirichlet series.

\begin{thm}
 Suppose   $ 0=\lambda_0<\lambda_1< \cdots$ satisfy the Dirichlet condition. Let $A_0=1$ and $A_n$, $n\geq 2$ be positive real numbers. Define
\begin{gather*}
 \Lambda(x) =\sum_{\lambda_n \leq x} A_n, \\ \intertext{and suppose that}
 \int_1^\infty \frac{\log \Lambda(x)}{x^2} dx < \infty. \\ 
\intertext{Then we have for each $H>0$ that}
 \lim_{N \to \infty} \min_{|a_n| \leq A_n} \int_0^H \abs{1+\sum_{n=1}^N a_n e^{-i \lambda_n t}} dt>0.
\end{gather*}
\end{thm}

\subsection{A vanishing result for Dirichlet series}

When we first proved a result like Theorem 1 and Theorem 2 we could not prove the results for all $H>0$, but rather for some sufficiently large $H$. This was presented at the Zeta-Function-Days in Seoul, September 2009. About one month later we managed to prove the result for any $H>0$. The key result is the following vanishing result for Dirichlet series on intervals:

\begin{vanlem} Any Dirichlet series that is identically zero on an interval of absolute convergence is identically zero on the complex plane. 
\end{vanlem}

\begin{proof}
 We may as well consider general Dirichlet series
\begin{gather} L(s)= \sum_{n=0}^\infty a_n e^{-\lambda_n s}, \\ \intertext{where we have the Dirichlet condition} \label{dircon}
0=\lambda_0< \lambda_1 < \lambda_2 \cdots
\end{gather}
 We first remark that the statement is trivial unless the interval lies on the abscissa of convergence, since  a Dirichlet series is analytic to the right of its abscissa of convergence and an analytic function vanishing on an interval is zero everywhere.

In the general case it is somewhat more difficult, but not much so. Suppose that the Dirichlet series $L(s)$  is absolutely convergent on $\Re(s)=\sigma$.
We remark that the Dirichlet series is bounded on the right half plane $\Re(s) \geq \sigma$.
Also it is analytic on $\Re(s)>\sigma.$ Let $\phi(z)$ be a holomorphic bijection mapping $|z|<1$  to $\Re(s)>\sigma$. Then
\begin{gather*}
 f(z)=L(\phi(z))
\end{gather*}
will be a bounded holomorphic function on the unit disc. By a classical theorem for the Hardy space $H^2(T)$ (see for example Rudin \cite[Theorem 17.18]{Rudin}) we have that $f(z)$ is non-vanishing almost everywhere on $|z|=1$. This implies that $L(s)$ cannot be zero on a set of positive measure (in particular not on an interval) on $\Re(s)=\sigma$.
\end{proof}
\begin{rem} An alternative way to prove the Vanishing Lemma is to use the logarithmic integral and a variant of Lemma 5. \end{rem}

What we  use to prove Theorem 2 for any $H>0$ rather than some $H>0$ is the following immediate consequence of the Vanishing Lemma.

\begin{lem} Let $\lambda_n$ fulfill the Dirichlet condition \eqref{dircon}, and let  $B_n$ be a sequence of positive numbers such that 
\begin{gather*}
\sum_{n=1}^\infty B_n<\infty. \\ \intertext{Then for any $H>0$ we have}
\inf_{|b_n| \leq B_n} \int_0^H \abs{1+\sum_{n=1}^\infty b_n e^{-\lambda_n i t}} dt=\delta>0.
\end{gather*}
\end{lem}

\begin{proof} 
This follows immediately from the fact that the set of Dirichlet series with $|b_n| \leq B_n$ is a compact set in $L^1(0,H)$ and hence the infinum must be attained, or in fact be a minimum. It can not be zero because that would violate the Vanishing Lemma. Hence it must be greater than zero.
\end{proof}

\begin{rem}
 Lemma 1 is not effective, i.e. we do not give an explicit estimate for the lower bound in terms of the $B_n$. This can be done however by the same proof method. We will further investigate this in \cite{Andersson15}.
\end{rem}

\subsection{Paley-Wiener's theorems}
We  also use the following theorems of Paley-Wiener (See  Paley-Wiener \cite{PalWien} or Koosis \cite{Koosis}):
\begin{lem} (Paley-Wiener)
 Suppose $S(x)$ is a positive increasing function such that
 \begin{gather*} \int_0^\infty \frac{\log S(x)dx}{1+x^2}<\infty.\\ \intertext{Then given any $\epsilon>0$ there exists an entire function $\phi(x)$ of finite type $\epsilon$ such that}
\phi(x) \leq \frac 1 {S(|x|)}, \qquad x \in \R. \end{gather*}
\end{lem}
\begin{lem} (Paley-Wiener)
Suppose $\phi(x)$ is an entire function of exponential type $A$ such that
\begin{gather*}
 \int_{-\infty}^\infty \abs{\phi(x)}^2 dx <\infty.
\end{gather*}
 Then the Fourier-transform $\hat \phi$ will have support on $[-A,A].$
\end{lem}
A  direct consequence of Lemma 2 and Lemma 3 is the following:
\begin{lem}
 Let $\epsilon>0$ and  suppose  $S(x)$ is a positive increasing function such that
 \begin{gather*} \int_0^\infty \frac{\log S(x)dx}{1+x^2}<\infty.\\
\intertext{Then there exists a continuous function $f$ with support on $[0,\epsilon]$ such that $\hat f(0) \neq 0$, and  such that}
 |\hat f(t)| \leq \frac 1 {S(|t|)}, \qquad t \in \R.
\end{gather*}
\end{lem}
\begin{proof}
  By Lemma 2 we can find an entire function $g$ of type $\epsilon/2$, such that  $|g(t)| \leq 1/S(|t|)$. We may assume that $g(0) \neq 0$, since otherwise we can consider the function
$$
  \tilde g(z)= c_0 \frac{g(z)}{z^n},
$$
where $n$ is the order of the zero of $g(z)$ at $z=0$, and $c_0>0$ is chosen small enough for   $|\tilde g(t)| \leq 1/S(|t|)$ to be valid for $-1 \leq t \leq 1$. By Lemma 3, the Fourier transform $\hat g(t)$  has support on $[-\epsilon/2,\epsilon/2]$. Thus   $f(t)= \hat g(t+\epsilon/2)$ has support on $[0,\epsilon]$ and  $\hat f(t)$ also fulfill the required inequality.
 \end{proof}

\subsection{Proof of Theorem 2}
Our proof will follow from our nonvanishing result for Dirichlet series, Lemma 1 (which in turn is a special case of Theorem 2), and the Paley-Wiener theorems in the form of Lemma 4:

From Lemma 1 we  find a  test  function $f(x)$ that is not the zero-function, with support on $[0,H/2]$ such that
\begin{gather*} |\hat f(x)| \leq \frac 1 {\Lambda(x)^2}, \qquad \text{and} \qquad \hat f(0) \neq 0.\end{gather*}
  Now consider the Dirichlet series
\begin{gather}
 B(s)= \hat f(0)+\sum_{n=1}^N a_n e^{-\lambda_n it} \hat f(\lambda_n) =
  b_0+ \sum_{n=1}^\infty b_n  e^{-\lambda_n it}.
\end{gather}
It is clear that $B(s)$ can be given by the convolution
\begin{gather} \label{conv}
  B(s)=\int_{0}^{H/2} A(s+ix) f(x) dx, \qquad \text{where}  \qquad  A(s)=1+\sum_{n=1}^N a_n e^{-\lambda_n s}.
\end{gather}
It follows that if
\begin{gather*}
 B_n =  \frac {A_n} {\Lambda(\lambda_n)^2} =\frac {A_n} {\p{\sum_{k=1}^n A_k}^2} ,  \\ \intertext{then}
 |b_n| \leq B_n.
\end{gather*}
By dyadic division we have
\begin{align*} 
\sum_{n=1}^\infty \frac{A_n}{\Lambda(\lambda_n)^2}&= \sum_{m=0}^\infty \sum_{2^m \leq  \Lambda(\lambda_n)<2^{m+1}}   \frac{A_n}
{\Lambda(\lambda_n)^2}, \\ &\leq  \sum_{m=0}^\infty \sum_{2^m \leq  \Lambda(\lambda_n)<2^{m+1}}  \frac{A_n}{(2^{m})^2}, \\ &\leq \sum_{m=0}^\infty \frac{2^{m}}{(2^m)^2} =2.
 \end{align*}
Thus we have 
\begin{gather}
  \sum_{n=1}^\infty B_n \leq 2<\infty,
\end{gather}
and we can apply Lemma 1 on the Dirichlet series $B(s)$ and the interval $[0,H/2]$. We have that
\begin{gather} \label{ab1}
  0 < \delta b_0 \leq \int_0^{H/2} \abs{B(it)} dt.
\\ \intertext{By \eqref{conv} we see that }
   \int_0^{H/2} \abs{B(it)} dt = \int_0^{H/2} \abs{\int_{0}^{H/2} A(it+ix) f(x) dx} dt. \label{ab2}
   \end{gather}
 By  \eqref{ab1},\eqref{ab2} and the triangle inequality we obtain 
\begin{gather*}
  0<\delta b_0 \leq  \int_0^{H} \abs{ A(it)} dt \int_0^{H/2} \abs{f(x)} dx.
 \end{gather*}
whenever $a_n \leq |A_n|$. Since $f(x)$ is not the zero-function and has support on $[0,H/2]$ we can divide the inequality with the right most integral and we get that
$$ 0<\frac{\delta b_0} { \int_0^{H/2} \abs{f(x)} dx} \leq   \int_0^{H} \abs{ A(it)} dt,
$$
 for any Dirichlet polynomial $A(s)$ such that $\abs{a_n} \leq A_n$.\qed

\section{Approximation by Dirichlet polynomials}

\subsection{Approximation theory for Fourier polynomials and Dirichlet polynomials}
\subsubsection{Classical theory}
In our first disproof of Ramachandra's conjecture \cite{Andersson} we used the fact that the Dirichlet polynomials $\sum_{n=2}^N a_n n^{it}$ can approximate any function in $L^2(0,H)$, and then we used the fact that $-1$ belongs to that class, in order to obtain the fact that
\begin{gather*}
 \int_0^H \abs{1+\sum_{n=2}^N a_n n^{it}}^2 dt < \epsilon,
\end{gather*}
for any $\epsilon>0$ and some Dirichlet polynomial $\sum_{n=2}^N a_n n^{it}$ depending on $\epsilon$.  The difference here compared to classical results from approximation theory is that we now have some estimates on the growth of the coefficients $a_n$.

 The theory of approximation by Dirichlet polynomials (and Fourier polynomials) has been extensively studied. Examples includes classical trigonometric series (Fourier theory) and  the Szasz-M\"untz theorem. A deep theorem that gives a quite satisfactory answer to the question of when a function on an interval can be approximated by complex exponentials is the Beurling-Malliavin theorem (see for example Koosis \cite{Koosis2}). For good surveys of this approximation theory, see the book of Levinson \cite{Levinson} and the  paper of Redheffer \cite{Redheffer}.
\subsubsection{Approximation with bounded coefficients}
However, when it comes to the corresponding approximation theory of Dirichlet (or Trigonometric) polynomials with bounded coefficients,  less has been done, and one of our aims in this paper is to make some advances in this theory. One simple example that shows that this approximation theory can be somewhat more difficult is the following example:
\begin{example}
  It is sufficient to show that $-1$ can be approximated by the Dirichlet polynomials $\sum_{n=2}^N b_n e^{-\lambda_n it}$ in $L^2(0,H)$ in order to prove that the Dirichlet polynomials are dense in $L^2(0,H)$.
\end{example}
\begin{proof}
 It follows by integrating both the constant $-1$ and the  Dirichlet polynomial $k$  times, that the polynomial $x^k$  can be approximated by  Dirichlet polynomials of the desired type. By Weierstrass theorem the polynomials $x^k$ are dense in $L^2(0,H)$ and it follows that  the Dirichlet polynomials are dense in $L^2(0,H)$.
\end{proof}
This proof does not work when we have conditions $|b_n| \leq B_n$. Instead  we need other methods.

\subsection{New approximation theorems}
Our main results about approximation by Dirichlet polynomials will be the following theorem:
 \begin{thm}
 Let  $A_n$ and $\Lambda(n)$ be defined as in Theorem 2. Suppose 
\begin{gather} \notag
 \int_1^\infty \frac {\varepsilon(x)}{{x}}dx<\infty, 
\\ \intertext{for some positive decreasing function $\varepsilon(x)$ and that}
 \label{condition}
   \Lambda(X) \ll \Lambda(X+Y)-\Lambda(X), \qquad (\varepsilon(X) \ll Y \ll 1)
\end{gather}
 for some $\delta>1$.  Let $H>0$. 
 Then the set of Dirichlet polynomials \begin{gather}\left  \{\sum_{n=2}^N a_n e^{-\lambda_n it}, |a_n| \leq A_n \right \} \notag \\  \intertext{is dense in $L^2(0,H)$ if and only if} 
 \int_1^\infty \frac {\log \Lambda(x)}{{x^2}}dx=\infty. \label{iiii}
\end{gather}
Furthermore the  conclusion holds true if we have the additional assumption that
\begin{gather*}
  \sum_{n=1}^\infty A_n^2 <\infty
\end{gather*}
is convergent and  $|a_n| \leq A_n$ is replaced by $|a_n| = A_n$. Also, under this assumption we may replace the set of Dirichlet polynomials with the set of convergent Dirichlet series in $L^2(0,H)$ such that  $|a_n| = A_n$.
\end{thm}

\begin{rem} We can for example choose $\varepsilon(x)=x^{-\delta}$ or $\varepsilon(x)=\log(x+1)^{-1-\delta}$ for some $\delta>0$ in Theorem 3, and these examples indeed seems sufficient for the applications we consider in this paper. In general it  is an interesting problem to try to replace \eqref{condition} with as weak condition as possible. Can we find some $\varepsilon(x)$ such that the integral is divergent but we still have Theorem 3?
 \end{rem}

\subsection{The Pechersky rearrangement theorem}

The fact that we  used universality to disprove Ramachandra's original conjecture, suggests that methods from universality should be used. We will here state a variant of the  Pechersky Rearrangement theorem (see Pechersky \cite{Pechersky}, Steuding \cite[Theorem 5.4]{Steuding},  Voronin \cite{Voronin} or  Bagchi \cite{Bagchi}) which is an important tool used to prove the Voronin universality theorem.

\begin{pecher} Let $\{x_n: n \geq 1 \}$ be a sequence in a complex Hilbert space $X$ satisfying: \begin{gather} \notag
    \sum_{n=1}^\infty \abs{\ip{x_n} {x}}=\infty, \qquad  \text{ for } \qquad x \in X, \, \, \, x \neq 0.
\\ \intertext{Then the set}  \notag \left \{\sum_{n=1}^m a_n x_n:  |a_n| \leq 1 \right \} \\ \intertext{is dense in $X$. If furthermore}
   \sum_{n=1}^\infty  \norm{x_n}^2< \infty. \label{sqr} \\ \intertext{Then the set}
  \left \{\sum_{n=1}^m a_n x_n:  |a_n| = 1 \right \},  \notag \\ \intertext{as well as the set of convergent series}
    \left \{\sum_{n=1}^\infty a_n x_n:  |a_n| = 1 \right \}, \notag
\end{gather}
are  dense in $X$.
\end{pecher}

\begin{proof}
  The last part of the result is exactly the Pechersky rearrangement theorem as given in Steuding \cite[Theorem 5.4]{Steuding}. 

The second part of the result is simpler to prove since the construction of a convergent element in the set \cite[p. 90]{Steuding} is not needed.

 The first part of the result is even easier to prove and follows from the same proof as the general case, see Steuding \cite[pp. 90-93]{Steuding}. It can be simplified considerably since the only time in the proof where the argument \eqref{sqr} is used, is when it is proved that we can choose $|a_n|=1$ instead of $|a_n| \leq 1$. Therefore, the arguments on p. 92-93 that use Lemma 5.2  in Steuding \cite{Steuding}  are not needed.
\end{proof}

\subsection{The Hilbert space of $L^2$-functions on an interval}

The difference in applying this theorem is that we use a different Hilbert-space than usual in universality. We use the simple Hilbert space $L^2(0,H)$ 
$$
 \ip f g = \int_0^H f(t) \overline{g(t)}dt,
$$
where the integral here is over an interval. In usual universality, it is over a two-dimensional set in the complex plane. This means that we use different theorems about entire functions, such as the Paley-Wiener's theorems instead of Bernstein's theorem. We also need the following Theorem (see Koosis,  \cite[pp. 49-50]{Koosis}) which is related to the previously stated Payley-Wiener theorems:

\begin{lem} Let $f(x)$ be an entire function of exponential type. Then
$$ \int_0^\infty \frac{\log^+ |f(x)|dx}{1+x^2}<\infty \implies \int_0^\infty \frac{\log^- |f(x)|}{1+x^2} dx<\infty.
$$
\end{lem}

\subsection{Proof of approximation theorems}

We are almost ready to  prove Theorem 3. We use our version of the Pechersky rearrangement theorem.
\subsubsection{Another lemma on entire functions of finite type}

 First we will prove a simple lemma that we will use.
\begin{lem}
  Suppose  $f(x)$ is a a continuous function with compact support and that $\hat f(0)=1$. 
If  $\inf_{\hat f(z)=0} |x-z| >\delta>2\varepsilon(x)$, then
\begin{gather}
 \min_{t \in[x,x+\varepsilon(x)]} \abs{\hat f(t)}=\log \abs{\hat f(x)}+\Oh{\delta^{-1} \varepsilon(x) x}.
\end{gather}
\end{lem}
\begin{proof}
  Since $f(t)$ is a continuous function with compact support then $\hat f(t)$ is an entire function of finite type. Since $\hat f(0)=1$ it will have the Hadamard product
  \begin{gather*}
   \hat f(z)=e^{az} \prod_{k=1}^\infty \p{1-\frac{z}{z_k}}.
  \end{gather*}
   By taking the logarithm of this we get 
    \begin{gather*}
   \log \abs{\hat f(z)}= \Re(az)+ \sum_{k=1}^\infty \log \abs{1-\frac{z}{z_k}}.
  \end{gather*} 
   Let $t=x+h$. We find that
   \begin{align*} 
\label{ah}
   \log \abs{\hat f(x+h)} - \log \abs{\hat f(x)} &=
 \Re(ah)+   \sum_{k=1}^\infty \log \abs{1- \frac{h}{z_k-x}}  \\ 
    &=  \Oh{h} + \sum_{k=1}^\infty \frac h {z_k-x} + \Oh{\sum_{k=1}^\infty \frac {h^2} {\abs{z_k-x}^2}}. 
\end{align*}
Since $\hat f(z)$ is an entire function of finite type and  thus by the Theorem \cite[p. 15]{Koosis} we have that
\begin{gather*}
 n(r) \leq  cr+\Oh{1},
\end{gather*}
for some $c>0$, where $n(r)$ denotes the number of zeroes of $\hat f(z)$ with $|z| \leq r$. From this it follows that 
\begin{gather*}
\abs{\sum_{k=1}^\infty \frac 1 {z_k-x}} \ll \delta^{-1} x, \qquad   \sum_{k=1}^\infty \frac 1 {\abs{z_k-x}^2} 
\ll \delta^{-2}x.
\end{gather*}
Our lemma follows by noticing that $0 \leq h \leq \varepsilon(x)<\delta/2$.
\end{proof}

\subsubsection{Proof of Theorem 3}
It is clear that Theorem 3 follows by  proving the  two cases
\begin{enumerate}
  \item If the integral \eqref{iiii} is convergent, then our set of Dirichlet series is not dense in $L^2(0,H)$.
   \item If the integral  \eqref{iiii} is divergent, then our set of Dirichlet series is  dense in $L^2(0,H)$.
\end{enumerate}
{\em Proof of Case 1. The integral  \eqref{iiii} is convergent.} By Theorem 2 it follows that there exists some $\delta>0$ such that if $$P(t)=\sum_{n=1}^N a_n e^{-\lambda_n it},$$ is any polynomial with coefficients $|a_n| \leq A_n$. Then
\begin{gather*}
  \int_0^H \abs{1+P(t)}^2dt \geq \delta. 
\end{gather*}
This means that the function $f(t)=-1$ cannot be approximated by such a Dirichlet polynomial in $L^2(0,H)$-norm, and  thus this set of Dirichlet polynomials with bounded coefficients  is not dense in $L^2(0,H)$. \qed
\vskip 4pt

\noindent {\em Proof of Case 2. The integral  \eqref{iiii} is divergent.} Let
\begin{gather*} x_{n}(t)=A_n e^{-\lambda_n it},
\end{gather*}
for $0 \leq t \leq H$. By the first part of our version of the Pechersky's rearrangement theorem it is sufficient to prove that
\begin{gather} \label{ajt}
\sum_{n=1}^\infty \abs{\ip x {x_n}}=\infty,
\end{gather}
for any non trivial function  $x=f(t)$  in $L^2(0,H)$, in order to prove the first part of Theorem 3. Since 
\begin{gather*}
  \sum_{n=1}^\infty \abs{\ip {x_n} {x_n}}=H \sum_{n=1}^\infty A_n^2,
\end{gather*}
we see last two statements in Theorem 3 corresponds to the last two statements in the Pechersky rearrangement theorem. Thus it is sufficient to prove \eqref{ajt} to prove the last two statements of Theorem 3 also. 

We thus proceed to prove \eqref{ajt}. It is clear that
\begin{gather} \label{andaj}
 \sum_{n=1}^\infty \abs{\ip x {x_n}}=
 \sum_{n=1}^\infty A_n  \abs{\int_0^H f(t) e^{-\lambda_n i t} dt}  = 2 \pi \sum_{n=1}^\infty A_n  \abs{\hat f(\lambda_n)}. 
\end{gather}
Let us choose 
\begin{gather} \label{deltadef}
  \delta= \frac 1 {8e H}.
\end{gather}
The integral condition assures that the  limit of the decreasing positive function $\varepsilon(x)$  is zero. Thus we can find some positive number $X_1$  so that $\varepsilon(X_1)<\delta/2$.  We will now disregard  the $\lambda_n<X_1$, in the sum \eqref{andaj}. This can be done if we are only interested in determining whether \eqref{andaj} is convergent, since that sum will be finite. 
By dividing the remaining sum into sub intervals 
$[X_k,X_{k+1}]$, such that   $X_{k+1}=X_{k}+ \varepsilon(X_k)$  for $k \geq 1$, we see that
$$\sum_{n=1}^\infty A_n |\hat f(\lambda_n)| \geq \sum_{k=1}^\infty \min_{X_k \leq \lambda_n \leq X_{k+1}}  |\hat f(\lambda_n)|  \sum_{X_k \leq \lambda_n \leq X_{k+1}} A_n. $$
By the condition \eqref{condition} this is greater than something of the order
\begin{gather*}
 \sum_{k=1}^\infty \p{\min_{X_k \leq t \leq  X_{k+1}}  \abs{\hat f(t)}} (X_{k+1}-X_k) \Lambda(X_k). \end{gather*}
By replacing the sum with an integral  this can be estimated  from below by
$$
 \int_{X_1}^\infty \Lambda(x) \min_{t \in [x,x+\varepsilon(x)]} |\hat f(t)| dx.
$$
Since $x \geq \log (x+1) /(1+x^2)$ for $x>0$, it follows that this can be estimated from below by
\begin{gather} \label{abaj}
 \int_{X_1}^\infty \frac{\log \p{1+\Lambda(x) \min_{t \in [x,x +\varepsilon(x)]} |\hat f(t)|}} {1+x^2} dx.
\end{gather}
Now, Let $\{z_k \}_{k=1}^\infty$ be the zeroes of the entire function $\hat f(z)$. Let us  consider \eqref{abaj}  for the case where $|z_k-x|>\delta$ for all zeroes $z_k$ of $\hat f$.   Since the integrand is positive, when estimating the integral from below we can discard the integral when $|z_k-x|<\delta$ for some $z_k$ such that $\hat f(z_k)=0$. By the fact that the logarithm-function is an increasing function we get the lower bound
\begin{gather*} 
 \int_{x \geq X_1, |x-z_k| > \delta} \frac{\log \p{\Lambda(x) \min_{t \in [x,x+\varepsilon(x)]} |\hat f(t)|}} {1+x^2} dx. 
\end{gather*}
By using the logarithm laws this integral divides into
\begin{gather} \label{abajajaj}
 \int_{x \geq X_1, \, \abs{x-z_k} > \delta} \frac{\log \p{\Lambda(x)}} {1+x^2} dx + 
  \int_{x \geq X_1, \, \abs{x-z_k} > \delta} \frac{\log \p{\min_{t \in [x,x+\varepsilon(x)]} |\hat f(t)|}} {1+x^2} dx.
\end{gather}
By Lemma 6 we see that the fact that $|z_k-x|>\delta>2 \varepsilon(x)$ implies that
\begin{gather*}
 \min_{t \in [x,x+\varepsilon(x)]} \log |\hat f(t)|  =\log |\hat f (x)|+ \Oh{\varepsilon(x) \delta^{-1} x}
\end{gather*}
This allows us to estimate the second integral in \eqref{abajajaj} with
\begin{gather} \label{cd1}
   \int_{x \geq X_1,|x-z_k|>\delta} \frac{ \log |\hat  f(x)|} {1+x^2} dx+  \Oh {\delta^{-1} \int_{x \geq X_1,|x-z_k|>\delta} \frac{\varepsilon(x)} x dx}.
\end{gather}
The second integral in Eq. \eqref{cd1} is finite by the condition on $\varepsilon(x)$ in the Theorem. The first integral can be estimated by Lemma 5, and is bounded. Thus it is sufficient to prove that
\begin{gather} \label{hud}
 \int_{|x-z_k|>\delta, x \geq X_1} \frac{ \log \Lambda (x)} {1+x^2} dx 
\end{gather}
is divergent.

By Theorem \cite[p. 15]{Koosis} we have
\begin{gather*}
 n(r) \leq  eHr+\Oh{1},
\end{gather*}
This means that there are a maximum of $2eHX+\Oh{1}$ zeroes of $\hat f(x)$ in the interval $[X,2X]$. Each zero will remove $2 \delta$ from the measure of the set of $x$ such that $\abs{x-z_k}>\delta$. In other words:
\begin{gather*}
  \int_{X\leq x \leq 2X, |x-z_k|>\delta} 1 \, dx \,  \geq  \, X-4eH \delta X + \Oh{1} \geq \frac X 2 + \Oh{1}.
 \end{gather*}
by the choice of $\delta$ in Eq. \eqref{deltadef}. Since $\Lambda(x)$ is an increasing function it follows from this inequality that the integral $\eqref{hud}$ can be estimated from below by a positive constant times
\begin{gather*} 
 \int_{x \geq X_1} \frac{ \log \Lambda (x)} {1+x^2} dx 
\end{gather*}
 Since this integral is divergent by the conditions in the theorem, the integral \eqref{hud} and  the sum in \eqref{ajt} that is bounded from below by this integral, are also divergent for any non trivial function $x=f(t)$ in $L^2(0,H)$. This concludes our proof of Theorem 3. \qed

\section{Generalized versions of the refined Ramachandra problem}

It is clear that Theorem 1 follows from Theorem 2 and Theorem 4. However, since we stated those theorems in terms of general Dirichlet series we will choose to state Theorem 1 in more generality as well.

\subsection{Generalized Dirichlet polynomials}

First we consider the case of generalized Dirichlet polynomials.
\begin{thm}
 Suppose  $A_n>0$ and that $\lambda_n$ fulfill the Dirichlet condition \eqref{dircon}. Let 
\begin{gather*}
  \Lambda(X)=\sum_{\lambda_n \leq X} A_n,
\qquad \text{and} \qquad  \int_1^\infty \frac{\varepsilon(x)} x dx <\infty\\ \intertext{for some positive decreasing function $\varepsilon(x)$ and suppose that}
    \Lambda(X+Y)-\Lambda(X) \gg Y \Lambda(X), \qquad \qquad (\varepsilon(X) \ll Y \ll 1),
 \end{gather*}
 for $X \geq X_0$. Then we have that for any $H>0$
 \begin{gather*}
   \lim_{N \to \infty} \min_{|a_n| \leq A_n} \int_0^H \abs{1+\sum_{n=1}^N a_n e^{-\lambda_n it}}^2 dt
>0,  \\ \intertext{if and only if}
    \int_1^\infty \frac {\log \Lambda(x)}{{x^2}}dx<\infty.  
  \end{gather*}
  The same conclusion holds true if $\min_{|a_n| \leq A_n}$ is replaced by $\min_{|a_n| = A_n}$ under the additional assumption that $$ \sum_{n=1}^\infty A_n^2<\infty.$$
 \end{thm}
\begin{proof}
 From Theorem 2 it follows that if 
 \begin{gather*} 
   \int_1^\infty \frac {\log \Lambda(x)}{{x^2}}dx<\infty, \\ \intertext{then}
\lim_{N \to \infty} \min_{|a_n| \leq A_n} \int_0^H \abs{1+\sum_{n=1}^N a_n e^{-\lambda_n it}} dt>0, \\ \intertext{and the lower bound follows from the Cauchy-Schwarz inequality. By Theorem 3 it follows that if}
\int_1^\infty \frac {\log \Lambda(x)}{{x^2}}dx=\infty, \\ \intertext{then the Dirichlet polynomials}
\sum_{n=1}^N a_n e^{-\lambda_n it}, \qquad |a_n|\leq A_n, \\ \intertext{are dense in $L^2(0,H) $. In particular it means that $f(t)=-1$ can be approximated by the Dirichlet polynomials and that for each $\epsilon>0$ there exists and  $N$ and $|a_n| \leq A_n$ such that}
  \int_0^H \abs{1+ \sum_{n=2}^N a_n e^{-\lambda_n it}}^2 dt<\epsilon.
  \end{gather*}
   Since $\epsilon$ can be chosen to be arbitrarily small this proves Theorem 4 in the case when the integral is infinite.
\end{proof}

\subsection{Classical Dirichlet polynomials}

We will now apply these results  on classical Dirichlet series: 
\begin{thm} Suppose that $A_n$ are positive numbers such that
\begin{gather*} 
\frac 1 M \sum_{n=T}^{T+M} A_n  \asymp \Phi(T), \qquad  \qquad T/(\log \log T)^{1+\delta} \leq M \leq T. \\ \intertext{for some $\delta>0$, $T \geq T_0$ and some positive increasing function $\Phi(n)$. Then}
\lim_{N \to \infty} \min_{|a_n| \leq A_n} \int_0^H \left|1+\sum_{n=2}^N a_n n^{it-1} \right|^2 dt > 0,  \\ \intertext{if and only if}
  \sum_{n=2}^\infty  \frac {\log{\Phi(n)}}{n \log^2 n}<\infty. 
\end{gather*}
\end{thm}

\begin{proof}
 This follows from Theorem 4 with $\lambda_n=\log (n+1)$, and  $\varepsilon(x)=(\log (x+1))^{-1-\delta}$. \end{proof}
\subsubsection{Proof of Theorem 1.}
Theorem 1 follows from using $A_n=\Phi(n)$ in Theorem 5. \qed

\subsubsection{Classical Dirichlet polynomials with arithmetical coefficients}

 We will  mention two other simple applications which also follows 
directly from Theorem 5. First for primes:
\begin{cor} Suppose  $\Phi(p)$ is an increasing positive function and $H>0$. Then 
\begin{gather*}\lim_{N \to \infty} \min_{|a_p| \leq \Phi(p)} \int_0^H \left|1+\sum_{\substack{p \text{ prime}. \\ p \leq N }} a_p p^{it-1} \right|^2 dt > 0,  \\ \intertext{if and only if}
  \sum_{p \text{ prime}}  \frac {\log{\Phi(p)}}{p \log p}<\infty. 
\end{gather*}
\end{cor}
We also mention the case of divisor coefficients:
\begin{cor} Suppose $\Phi(n)$ is an increasing positive function and $H>0$. Then 
\begin{gather*}\lim_{N \to \infty} \min_{|a_n| \leq \Phi(n)} \int_0^H \left|1+\sum_{n=2}^N  a_n d(n) n^{it-1} \right|^2 dt > 0  \\ \intertext{if and only if}
  \sum_{n=2}^\infty  \frac {d(n) \log{\Phi(n)}} {n (\log n)^3}<\infty. 
\end{gather*}
\end{cor}
In fact most known cases of arithmetic functions can be used as coefficients, divisor coefficients $d_k(n)$, Fourier coefficients of cusp-forms and so forth.

\subsection{Shifted classical Dirichlet polynomials}
We will also show a theorem that has applications on the Hurwitz zeta-function.
\begin{thm} Suppose  $\alpha>0$, and that $\Phi(n)$ is an increasing positive function and $H>0$. Then 
\begin{gather*}\lim_{N \to \infty} \min_{|a_n| \leq \Phi(n)} \int_0^H \left|\alpha^{it-1}+\sum_{n=1}^N a_n (n+\alpha)^{it-1} \right|^2 dt > 0,  \\ \intertext{if and only if}
  \sum_{n=1}^\infty  \frac {\log{\Phi(n)}}{n \log^2 n}<\infty. 
\end{gather*}
\end{thm}
\begin{proof}
  This follows from Theorem 4 with $\lambda_n=\log(n+1+\alpha)-\log(\alpha)$. 
 \end{proof}

\section{Applications on zeta-functions}

\subsection{Application on the Riemann zeta-function}

\begin{thm}
  Suppose   $\omega (t)\leq 1$ is an increasing function
such that
$$\int_2^\infty \frac{1-\omega(t)}{t \log t} dt < \infty.$$
 Then for each $\delta>0$ there exists a $C_\delta>0$ such that
$$
 \int_{T}^{T+ \delta} |\zeta(\sigma+it)| dt \geq C_\delta, \qquad  \omega(T) \leq \sigma.
$$
\end{thm}
\begin{proof}
The idea is simply to approximate $\zeta(s)$ with the Dirichlet polynomial (Ivic \cite{Ivic}, Theorem 1.8)
\begin{gather}\label{y2} \zeta(\sigma+it)=\sum_{1 \leq n <T} n^{-it-\sigma} + o(1),\end{gather}
where $$n^{1-\omega(T)} \ll \Phi(n)=e^{(1-\omega(n)) \log(n)},$$ and then use Theorem 1. We remark that while Theorem 1 is stated in $L^2$-norm, in the case when the sum is 
is convergent and the integral is bounded from below by a positive constant, the theorem is true also in $L^1$-norm. This is because Theorem 2 is true in $L^1$-norm, and it is Theorem 2 that is used to prove the lower bound in Theorem 1.
\end{proof}

Examples include the following cases: 
\begin{gather*}
 \int_{T}^{T+ \delta} \abs{\zeta \p{1- (\log \log T)^{-1-\epsilon}+it}} dt \geq C_\delta, \\ 
\int_{T}^{T+ \delta} \abs{\zeta \p{1-(\log \log \log T)^{-1-\epsilon} (\log \log T)^{-1}+it}} dt \geq C_\delta, \\ \intertext{and}
\int_{T}^{T+ \delta} \abs{\zeta \p{1 -(\log \log \log \log T)^{-1-\epsilon}(\log \log \log T)^{-1} (\log \log T)^{-1}+it}} dt \geq C_\delta.
\end{gather*}

By using some ideas of this paper, but by using a test-function of Ramachandra instead of  the test-functions of Paley and Wiener, we will be able to improve on Theorem 7. For details,  see our forthcoming paper \cite{Andersson4}, where we show that we can choose $\omega(T)=1-C\delta/\log \log T$, for any positive constant $C<\pi/4$.

\subsection{Applications on other Dirichlet series}

We remark that  the corresponding results for the the Dirichlet L-functions (and any element in the Selberg class as well) also follows by the same argument. We will  state the Theorem that corresponds to Theorem 7 for the Hurwitz zeta-function.

\begin{thm}
  Suppose   $\alpha>0$ and that $\omega (t)\leq 1$ is an increasing function
such that
$$\int_2^\infty \frac{1-\omega(t)}{t \log t} dt < \infty.$$
 Then for each $\delta>0$ there exists a $C_\delta>0$ such that
$$
 \int_{T}^{T+ \delta} |\zeta(\sigma+it,\alpha)| dt \geq C_\delta, \qquad  \omega(T) \leq \sigma.
$$
\end{thm}
\begin{proof}
 This follows from Theorem 6 in the same way as Theorem 7 follows from Theorem 1.
\end{proof}
We remark that Theorem 8 also can be stated and proved for the Lerch zeta-function (an analogue of eq. \eqref{y2} is true for the Lerch zeta-function also  \cite[Theorem 1.2]{GarLau}), and it follows from Theorem 6 by the same method.
 We also remark here that this result is important since the Hurwitz zeta-function does not have an Euler-product and our other methods that requires Euler products (e.g. \cite{Andersson3}) do not apply. 
 \begin{cor} Let $\alpha>0$. Then
\begin{gather*} 
  \int_{T}^{T+\delta} \abs{\zeta\p{1+it,\alpha}}dt \geq C_\delta>0,
\end{gather*}
\end{cor}
This is the first result that shows a a lower bound in this problem for the Hurwitz zeta-function that is independent of $T$, and thus that the Hurwitz zeta-function is not  universal on the line $\Re(s)=1$. For the special case of the Riemann zeta-function and  for an argument for why this implies non universality and what universality on a line means, see the discussion in \cite[pp. 5-6]{Andersson3}.

\bibliographystyle{plain}

\end{document}